\documentclass[10pt,draft,reqno]{amsart}
      \makeatletter
     \def\section{\@startsection{section}{1}%
     \z@{.7\linespacing\@plus\linespacing}{.5\linespacing}%
     {\bfseries
     \centering
     }}
     \def\@secnumfont{\bfseries}
     \makeatother
\newtheorem{theorem}{Theorem}[section]
\newtheorem{lemma}[theorem]{Lemma}
\newtheorem{proposition}[theorem]{Proposition}

\theoremstyle{definition}

\theoremstyle{remark}

\numberwithin{equation}{section}
\def \e{{\varepsilon}}

\def \k{{\kappa}}
\def \l{{\lambda}}

\def \p{{\varphi}}

\def \m{{\mu}}

\def \qq{{\qquad}}
\def \R{{\bf R}}

\def \dd{{\rm d}}

\def \noi{{\noindent}}

\def\R{{\mathbb R}}

\font\gum= cmbx10 at 13 pt
\font\phh=cmcsc10
at  8 pt
\scrollmode

   at 10,8  pt\font\phh=cmr10 at  8,2 pt
  \scrollmode
 at 9,5  pt
\font\gssec= cmb10 at 8,2  pt

   \title[Means values of  Fourier transforms]{A bound for Mean  values of  Fourier transforms}
  \author{ Michel J.\,G. Weber}
\address{ Michel Weber: IRMA, Universit\'e
Louis-Pasteur et C.N.R.S.,   7  rue Ren\'e Descartes, 67084
Strasbourg Cedex, France. }
\email{michel.weber@math.unistra.fr
}
 \urladdr{http://www-irma.u-strasbg.fr/$\sim$weber/}
   
\begin{document}
\maketitle 
\centerline{-----------------------------------------------------------------------------------------------------------} 
     \begin{abstract}    We show that there exists a sequence $\{n_k, k\ge 1\}$ growing at least geometrically
such that for any     finite non-negative    measure $\nu$ such that  
$\widehat
\nu\ge 0$, any $T>0$,  
$$ \int_{ -2^{n_k } T}^{2^{n_k } T }   
\widehat \nu(x)   \dd x \ll_\e  T\,2^{2^{(1+\e)n_k}}       
\int_\R      
 \Big|{\sin  {  xT } \over
 xT } \Big|^{ n_k^2 }   
 \nu(\dd x).
$$    \end{abstract}
 
      \vskip 0,5cm \noi {\gssec 2010 AMS Mathematical Subject Classification}: {\phh Primary: 60F15, 60G50 ;
Secondary: 60F05}.  \par\noi  
{{\gssec Keywords and phrases}: {\phh Fourier transform, mean value, convolution products, dilation, shift.}} 
\vskip 5pt
\centerline{-----------------------------------------------------------------------------------------------------------}

\def \noi{{\noindent}}
  \def \bt{{\hbox{\vrule  height 1pt depth 1pt width 1pt}}}  
\font\gum= cmbx10 at 13 pt
\font\phh=cmcsc10
 \scrollmode
\hfuzz =5 pt


\section{Introduction}
 Let  $\nu$ be   a finite non-negative    measure on $\R$,  $\widehat \nu(t)=\int_\R e^{itx} \nu(\dd x ) $,  then
$$  {1\over  T}\int_{ -T}^{ T}\widehat \nu (t) \dd t=\int_{\R} 
 {\sin
  Tu\over
  Tu  }    \, \nu( \dd u) .$$
Assume $\widehat \nu\ge 0$,  then 
 \begin{equation} \label{21} \Big|  \int_\R \Big({\sin
  T u /2\over
    Tu /2  }\Big)^2\nu (\dd u) \Big| \le {1\over T}\int_{-T}^T \widehat \nu  (x)  \dd x
\le  3\int_{\R} 
 \Big({\sin
   Tu/2\over
   Tu/2  }\Big)^2  
\nu (\dd u).
\end{equation}
 The first inequality is in turn  true at any order: for   any positive integer
$\kappa $,
 \begin{equation} \label{21nu} \Big|  \int_\R  \Big({\sin
  T u /2\over
   T u /2  }\Big)^{2\kappa } \nu (\dd u) \Big|
\le  {1\over T}\int_{-\kappa T}^{ \kappa T}  |\widehat \nu  (x)| \dd x.  
\end{equation}  
 
The question  whether the second inequality admits a similar extension   arises naturally. 
We show the existence of
   a general form of that inequality
 in which appear constants growing fastly with   $\kappa $.  
 \begin{theorem}\label{cc} There exists a sequence $\{n_k, k\ge 1\}$ growing at least geometrically
such that for any     finite non-negative    measure $\nu$ such that  
$\widehat
\nu\ge 0$, any $T>0$, we have
$$ \int_{ -2^{n_k } T}^{2^{n_k } T }   
\widehat \nu(x)   \dd x \ll_\e  T\,2^{2^{(1+\e)n_k}}       
\int_\R      
 \Big|{\sin  {  xT } \over
 xT } \Big|^{ n_k^2 }   
 \nu(\dd x).
$$
  \end{theorem} 
 We don't know whether the constant $2^{2^{(1+\e)n_k}}$ can be significantly weakened. The proof of   is rather delicate. 
   In order to prepare it, and also  to provide the necessary hints concerning inequalities (\ref{21}),(\ref{21nu}),   
  introduce some auxiliary functions and indicate as well some related  properties. Let   
  $ K (t)    =\big( 1-|t| ) ^+$, 
$T>0$ and define 
 $   K_T(t) =K({t/ T})
 =\big( 1-|t|/T)\chi_{\{|t|\le T\}}
$. Then 
$$\widehat{ K} (u)  = \big({\sin    u/2\over  u/2  }\big)^2, \qq \widehat{ K}_T(u)  ={1\over T}\big({\sin   Tu/2\over 
u/2  }\big)^2    .$$

It is easy to  check that
 $K_T(t )+K_T(t +T)+ K_T(t -T)=1 $,
if $|t|\le T$. It  follows that
 \begin{eqnarray}\label{KT}
  \chi_{\{|t-H|\le T\}} \le K_T(t-H)+K_T(t-H+T)+ K_T(t-H-T)
  .
\end{eqnarray}
This can be used to prove (\ref{21}). Since $\int_{\R}  K_T(t-S) \widehat{ \nu }(t)\dd t = \int_{\R}  e^{iSx} \widehat{K_T}(x)    \nu (\dd x) $, we deduce
  \begin{eqnarray*} \int_{H-T}^{H+T}\widehat{ \nu}(t) \dd t &\le & \int_{\R} \Big[K_T(t-H)+K_T(t-H+T)+ K_T(t-H-T)\Big]
\widehat{
\nu }(t) \dd t \cr &\le & \int_{\R} \widehat{K_T}(u) \Big[e^{ iHu} +e^{ i(H-T)u}+e^{
i(H+T)u}\Big]  \nu (\dd u)
\cr & =& {1\over T}\int_{\R}  \big({\sin   Tu/2\over 
u/2  }\big)^2  e^{ iHu}  [1+ 2\cos Tu]   \nu (\dd u).
\end{eqnarray*}
This immediately implies the second inequality in (\ref{21}).   Notice also by using Fubini's theorem, that    for any reals 
$S,\gamma, T$ reals, 
$T>0$    and any integer
$\kappa>0$,
 \begin{eqnarray}\label{KT1}
     \int_\R   \Big({\sin
  T(u-\gamma)/2\over
  T (u-\gamma)/2  }\Big)^{2 \kappa}  e^{iSu}\nu (\dd u)  ={1\over T}\int_\R e^{  -i \gamma (y-S)}K^{*\kappa}\Big(\frac{ y-S }{ T }\Big)
\widehat \nu
 (y) \dd y.
  \end{eqnarray} 

  Letting $\k=1,\gamma=S=0$, gives  $$   {1\over T}\int_\R \Big({\sin
   Tu/2\over
    u/2  }\Big)^2 \nu (\dd u)  =\int_{\R }K_T(y )\widehat \nu
  (y) \dd y    . 
 $$
As $0\le K_T(y )  =\big( 1-|y |/T)\chi_{\{|y |\le T\}} \le \chi_{\{|y |\le T\}}$, we deduce 
$$  \Big| {1\over T}\int_\R \Big({\sin
   Tu/2\over
    u/2  }\Big)^2 \nu (\dd u) \Big|\le \int_\R  K_T(y )|\widehat \nu
  (y)| \dd y    \le \int_{ -T}^{ T} |\widehat \nu (y)| \dd y, 
 $$
which yields  the first inequality in (\ref{21}). 
\vskip 2pt
As to (\ref{21nu}), some   properties of basic convolutions products are needed. Consider for $ A>0$ the    elementary measures  
$\m_A$ with density  
 $ g_{\m_A }(  x)= \chi_{\{[-A,A]\}}(x)$.  
  Let $0<A\le B$. Plainly
\begin{eqnarray}
 \label{convel} 
g_{\m_A*\m_B}(x)=g_{\m_A}*g_{\m_B}(x)=2A\big( 1-{|x|\over A+B}\big)\cdot \chi_{\{[ -A-B, A+B]\}}( x). 
\end{eqnarray}
Indeed,  $g_{\m_A*\m_B}(x) = \int_\R g_{\m_A }(x-y)g_{ \m_B}(y)\dd y =\int_{-B}^B\chi_{\{[x-A,x+A]\}}(
y)\dd y$, and this is equal to 
$$\l\big([ -B, B]\cap [x-A,x+A]\big)   =   \big[B\wedge(x+A)-(x-A)\vee(-B)\big]\chi_{\{[ -A-B, A+B]\}}( x).$$
In particular, introducing the function $     {\bf g} (x)= \chi_{[-{1\over 2},{1\over 2}]}(x)  $, we have $ K (t)={\bf g}*{\bf g}(t)  $.
\vskip 2pt
More generally,
\begin{lemma} \label{convelem}Let $0<A_1\le A_2\le \ldots
\le A_J$ and $\m=\m_{A_1} *\m_{A_2} *\ldots*\m_{A_J}$. Then  $\m$ has density
$g$ satisfying
$$0\le g(x)\le  G_J  \cdot\chi_{\{[-(A_1+A_2+\ldots
+A_J),A_1+A_2+\ldots +A_J]\}}(x),$$
where  
$$G_J= 2^J A_1  \cdot\big( (A_1+A_2)\wedge A_3\big)  \cdot\big( (A_1+A_2+A_3)\wedge A_4\big)\ldots \big( (A_1+\ldots+A_{J-1})\wedge
A_J\big).$$ 
\end{lemma}
\begin{proof} We prove it by induction.   
By (\ref{convel}), for every real $x$
$$0\le g_{\m_{A_1}*\m_{A_2}}(x)\le 
 2A_1 \cdot \chi_{\{[ -A_1-A_2, A_1+ A_2 ]\}}( x)=
2A_1\cdot g_{\m_{A_1+A_2}}(x).  $$
The case $J=2$ is proved. Now for $J=3$, by what preceeeds
\begin{eqnarray*}g_{\m_{A_1}*\m_{A_2}*\m_{A_3}}(x)&=& \int_{-A_3}^{A_3}g_{\m_{A_1}*\m_{A_2}}(x-y)dy\le
2A_1\,\int_{-A_3}^{A_3}g_{\m_{A_1+A_2}}(x-y)dy\cr &=&2A_1\, g_{\m_{A_1+A_2}*\m_{A_3}}  
\cr &\le& 2A_1\, 2( A_3 \wedge A_1+A_2)\cdot
\chi_{\{[ -A_1-A_2-A_3, A_1+A_2+ A_3 ]\}}( x).
\end{eqnarray*}  The general case follows by iterating the same argument.
\end{proof}
 
In particular,
for any positive $J$,
 \begin{equation}\label{kappaj} 0\le K^{*J}(x)\le \chi_{\{[- J, J]\}}(x) .
\end{equation}    
Indeed,   apply  Lemma \ref {convelem} with $A_j\equiv 1/2$. We get 
$$0\le K^{*J}(x) = {\bf g}^{*2J}(x)\le  G_{2J}  \cdot\chi_{\{[- J, J]\}}(x),$$
and  $G_{2J}= 2^{2J}\cdot 2^{-2J}=1  $.
Inequality (\ref{21nu})  is yet a direct consequence of (\ref{KT1}) and (\ref{kappaj}).

\vskip 4pt \par  Now, we pass to the preparation of the proof of Theorem \ref{cc}, and begin to explain how we shall proceed. By using
(\ref{KT}) with
$H=0$,
 $T=1/2$, we get 
 \begin{eqnarray} \label{conv01}
   {\bf g}(x)& \le & {\bf g}^{*2} (2x)  +  {\bf g}^{*2} (2x+1)  + {\bf g}^{*2} (2x-1)   . 
  \end{eqnarray}

 An important intermediate step towards the proof of Theorem \ref{cc} will consist to generalizing that inequality. Our approach can
be described as follows. As
$ 
 {\bf g}^{*2}(2v)  =  \int_\R {\bf g}(2v-y){\bf g}(y) \dd y
 $,  (\ref{conv01}) can be used  to   bound     the integration term ${\bf g}(y)$.  And by next reporting this into (\ref{conv01}), it
follows that
${\bf g}(x)$ can   also be bounded by a sum of terms of type
${\bf g}^{*3}$. 
   Call $\mathcal E$  this    operation.     By iterating $\mathcal E$, we  similarly  obtain    variant forms   of 
(\ref{KT}), involving higher     convolution powers of ${\bf g}$. The study of the iterated action of    $\mathcal E $, as well as
 the order of the constants generated  is made in the next section. The action of    $\mathcal E $ will be first  described as 
the combination of two elementary transforms  acting alternatively.

  \section{Stacks and  shifts} 
   
We first introduce    some operators and related auxiliary results, as well as the necessary notation. 
       Given   $f:\R\to \R$ and $a>0$, let $T_af (x)=f({x\over a})$ be   the dilation  of $f$ by $a^{-1}$.
  Plainly $T_a
T_b= T_{ab}$. 
  Notice also that 
\begin{eqnarray}\label{conv0}T_a (h*f)  ={1\over a}\, T_ah  *T_af ,   \qq\qq  h\in L^1(\R),\, f\in L^\infty(\R)     
.
\end{eqnarray} 
Indeed
\begin{eqnarray*}T_a (h*f )( u)&= &  \int_\R h({u\over a}- x) f(x) \dd x= \int_\R T_ah( u-ax)
T_af (ax)
\dd x \cr &= &{1\over a}\int_\R T_ah( u-v) T_af(v) \dd v= {1\over a}T_ah  *T_af (u) .
\end{eqnarray*} 
 
More generally
\begin{eqnarray}\label{convmany}T_a (h_1*\ldots *h_n)  ={1\over a^{n-1}}\,T_a (h_1)*\ldots *T_a (h_n) , \qq  h_1 ,   \ldots ,h_n\in
L^\infty(\R)  .
\end{eqnarray}
Introduce also  the  sequence of ${\bf g}$-dilations  $$ {\bf g} _{k} = T_{2^{-k}}{\bf g}, \qq\quad
k=1,2,\ldots
  .$$  

Now let  $I$ be a finite subset of $\R$.   It will be  convenient to denote   
\begin{equation}\label{Sigma}\Sigma [f(x):\!I] =\sum_{\rho\in I}f(x+\rho).
\end{equation}
  We have 
\begin{equation}\label{Sigma1}  \Sigma [f(bx): I]
 =\Sigma [T_{1\over b}f (x) :
{1\over b}I] .
\end{equation} 
And 
\begin{equation}\label{Sigma11}\widehat{\Sigma }[f \!:\!I](t)
=  \widehat
f(t)\Big(\sum_{\rho\in I}  e^{-it \rho } \Big) . 
\end{equation} 
  The  linear  operator $ f\mapsto\Sigma [f :\! I]$ on $L^1(\R) $  commutes with the convolution
operation:
\begin{equation}\label{Sigma2}f* \Sigma [h:\! I]= \Sigma [f*h:\! I] .
\end{equation} 
Further    $\Sigma [f :\!I]\ge 0$ if $f\ge 0$. We  
use the standard arithmetical set  notation: 
$\l I=\{\l\rho:  
\rho\in I\}$ and if $I, J$ are two finite subsets,
$I+J=\{\rho+\eta: \rho\in I, \eta\in J\}$, repetitions are counted. 
  This is   relevant      since
\begin{equation}\label{Sigma3}\Sigma\big[\Sigma [f(x):\!I]:\! J\big]=\Sigma [f(x):\!I+J].
\end{equation} 

Let $j_0<j_1\ldots<j_k$ be a finite set of positive integers, which we denote $J$. Let $C=\{c_j, j\in J\}$ be  
some  other set  of positive
  integers, not necessarily distinct.   We identify $(J,C)$ with 
 ${\bf U} := \{(j, c_j) , \ j\in J\} $, and
     put
\begin{equation}J^*=\begin{cases} J  &   {\rm if}  \  c_{j_0}>1 \cr 
 J\backslash\{j_0\} &   {\rm if}  \  c_{j_0}=1 .
\end{cases}
\end{equation}  
  
   Define the transform  $J 
 \to J_1$ as follows
\begin{equation}J_1= \mathcal D (J ):=J^*\cup (j_0+J) .\end{equation}
Next define  $C\to C_1$ by putting 
\begin{equation}C_1=\mathcal T (C):=\{c^1_j, j\in J_1\}, \end{equation} 
where 
\begin{equation}
c^1_j=\begin{cases} c_j+c_{j-j_0}  &   {\rm if}  \  j\in J^*\cap (j_0+J) \cr
 c_{j-j_0}  &   {\rm if}  \  j\in(J^*)^c\cap (j_0+J)   \cr 
c_j   &   {\rm if}  \  j\in J^*\cap (j_0+J)^c    \   {\rm and}  \   j>j_0 \cr
  c_{ j_0-1}  &    {\rm if}  \  j_0\in J^*
  .
\end{cases}
\end{equation}  
 
 Similarly we identify $(J_1,C_1)$ with 
 ${\bf U}_1 := \{(j, c^1_j) , \ j\in J_1\} $ The successive transforms 
 $(J,C)\!\to\!(J_1,C_1)\!\to\!(J_2,C_2)\!\to\! \ldots $ turn up to describe the iterated of $\mathcal E$, and  
   may be   compared  to the action of superposing   shifted functions.
 We start with $J=\{1\}  ,C=\{2\}$ corresponding to the basic set 
$${\bf U}=\{(1,2)\}.$$ It is easy to check that the iterated  transforms of ${\bf U}$
  progressively generate  the sequence  of sets
 \begin{eqnarray*} 
&   & \!\!\!\!\!\!\!\!\!\!\!\! \!(1,1), (2,2)    \cr 
    &   &\!\!\!\!\!\!\!\!\!\!\!\!\!\!\!\! \quad\quad\ \  \ \, (2,3), (3,2)  \cr 
  &   &\!\!\!\!\!\!\!\!\!\!\!\!\!
\quad\quad\ \   (2,2),  (3,2),  (4,3),  (5,2)   \cr
    &   &\!\!\!\!\!\!\!\!\!\!\!\!\!\!\!\! \quad\quad\ \    \   \, (2,1),  (3,2),  (4,5),  (5,4), (6,3), (7,2) 
\cr   &   &\!\!\!\!\!\!\!\!\!\!\!\!\!\!\!\!    \quad\quad\ \  \ \qquad\ \     \,   (3,2),  (4,6),  (5,6), (6,8), \, \ (7,6), \ \, (8,3),\
 \!  (9,2) 
\cr   &   &\!\!\!\!\!\!\!\!\!\!\!\!\!\!\! \!\!\!\!\!    \quad\quad\ \ \qquad\ \   \  \  \  \,     (3,1),  (4,6),  (5,6), (6,10), (7,12),
(8,9),(9,10),(10,6),(11,3),   
 (12,2) 
\cr   &   &\!\!\!\!\!\!\!\!\!\!\!\!    \qquad\qquad\qquad  \ldots\end{eqnarray*}
  
  At the $m$-th step, the set $J_m$ is an interval of integers   $\{a_m,\ldots, b_m\}$ with $a_m\to \infty$ slowly,  whereas
$b_m\to
\infty$  very rapidly. More precisely, let for $k=1,2, \ldots$   
$$   r_k= \max_{m\ge 1} c^m_k . $$
Then $r_1=2, r_2=3,r_3=2, r_4=6,\ldots    $ etc. And define  
\begin{equation}\label{Rz}R_k= r_1+\ldots+r_k, \qq \zeta_k=1+r_1+2r_2+\ldots+kr_k .
\end{equation}
Let $R_{k-1}<m\le R_k$. At step $m$, $J_m$ is realized by first shifting $J_{m-1}$ on the right from a length $ k$, next taking union with $J^*_{m-1}$
and in turn
$$J_m=\big\{k,k+1,\ldots, \zeta_{k-1} + (m-R_{k-1})k\big\},$$
 if $m<R_k$, whereas
 $J_{R_k}=\big\{ k+1,\ldots, \zeta_{k } \big\} $.  
\smallskip\par
Write $m=R_{k-1}+h$, $1\le h\le r_k$. Then we have the relations
\begin{equation}   c_{ j}^{m }=\begin{cases}  c_{ j-k}^{R_{k-1} } + [r_k+\ldots+(r_k-h+1)] &\qq  2k\le j \le \zeta_{k-1}  + (m-R_{k-1})k  , \cr 
   c_{ j}^{R_{k-1} }   & \qq  k\le j<2k.  
\end{cases}
\end{equation}
After the steps   
 $ R_{k-1}+1, R_{k-1}+2,\ldots,R_{k }$, the function  $h\mapsto c_{n}^{R_{k-1}+h }$ will have increased from 
$$r_k+ (r_k-1)+\ldots +2+1={r_k(r_k+1)\over 2}$$
for all 
 $n\in  \big\{ 2k, \ldots, \zeta_{k-1}\big\}$. It follows that 
\begin{equation} \label{minrn}\min_{n\in   \{ 2k, \ldots, \zeta_{k-1} \}}r_{n}\ge {r_k^2\over 2}.
\end{equation}
Therefore $r_{2k}\ge {r_k^2/ 2}$. This being true for all $k$,  yields by iteration  
$$r_{2^j}\ge {1\over 2^{}} (r_{2^{j-1}})^2 \ge {1\over 2^{}}{1\over 2^{2}} (r_{2^{j-2}})^{2^2}\ge \ldots\ge {1\over
2^{1+2+\ldots+2^{H-1}}} (r_{2^{j-H}})^{2^H} =\big({r_{2^{j-H}}\over 2}\big)^{2^H}.$$
We have $r_2=3$. Thus 
\begin{equation} r_{2^j}\ge \Big({3\over 2}\Big)^{2^{j-1}},  \qq\qq j=1,2,\ldots.
\end{equation}  
We shall deduce from this and (\ref{minrn}) that $r_k$ grows at least geometrically. 
Let $n$ and let $j$ be such that $2^{j+1}\le n<2^{j+2}$. Apply (\ref{minrn}) with $k=2^j$. As $n\ge 2k$, we have  $r_{n}\ge {r_{2^j}^2/ 
 2}$ once $2^{j+2}\le \zeta_{2^j-1}  $. But
$$\zeta_{2^j-1}\ge \zeta_{2^{j-1}} =1+r_1+2r_2+\ldots+2^{j-1}r_{2^{j-1}}\ge 2^{j-1}\Big({3\over 2}\Big)^{2^{j-2}}\gg  2^{j+2}.$$
  Thereby, for $j$   large 
$$r_n\ge {1\over 2} r_{2^j}^2 \ge  {1\over 2} \Big({3\over 2}\Big)^{2^{j-1}}={1\over 2} \Big({3\over 2}\Big)^{{2^{j+2}\over 8}}\ge
{1\over 2} \Big({3\over 2}\Big)^{{n\over 8}}={1\over 2}e^{({1\over 8}\log {3\over 2})n}  .$$
 Consequently, there is a numerical constant $\rho>1$, such that for all   $n\ge 1$,  we have
\begin{equation} \label{roestim} r_{n}\ge \rho^n.
\end{equation}  
 Let $j_m:=\#\{J_m\}  $. Since $j_m  =  \zeta_{k-1}+ (m-R_{k-1}-1)k  $ if $R_{k-1}< m\le R_k$,
we have  
$$\sum_{R_{k-1}< m\le R_k} j_m= \sum_{R_{k-1}< m\le R_k}(\zeta_{k-1}+ (m-R_{k-1}-1)k) 
  =r_k\zeta_{k-1}+k\sum_{u=1}^
{r_k-1}  u   .$$
We thus notice   for later use that
\begin{equation}\label{est0}\sum_{R_{k-1}< m\le R_k} j_m=r_k\zeta_{k-1}+k 
 {r_k (r_k+1)\over 2}   .
\end{equation} 
Let $ \displaystyle{\prod^*_{ j  } f_j}$  denotes the convolution product of   $f_j$'s.  Finally we  put 
\vskip -5pt $${\bf I} =\Big\{  {-1\over 2}, 0,{1\over 2}\Big\}
  .$$ 
   Our next result generalizes   inequality (\ref{KT}) to arbitrary convolution powers of
${\bf g}$.
 \begin{proposition} \label{p5}Let $k\ge 1$ and $R_{k-1}<m\le R_k$. Then 
$${\bf g}(x) \le C_m\ \Sigma\big[\prod^*_{(j,c_j)\in {\bf U}_m}{\bf g} _{j}^{*c_j} (x) : {\bf I}_m \big] , 
$$
where ${\bf I}_m,C_m $  are defined by the recurrence relations:  ${\bf I}_0=\big\{  {-1\over 2}, 0,{1\over 2}\big\}$, $C_0=2$ and
$$  {\bf I}_m={\bf I}_{m-1}+ r_k\, {\bf I}_{m-1}, \qq\qq C_m=2^{ k(j_{m-1} -1)}C_{m-1}^2  . $$
 
\end{proposition}
 \begin{proof} We   use repetitively the relation (see (\ref{conv0}))
$$T_{1\over 2}(h*f)  =2\, T_{1\over 2}h  *T_{1\over 2}f ,$$ $f\in L^\infty(\R),\, h\in L^1(\R)$.
   By (\ref{conv01}),  
 \begin{eqnarray} \label{conv1}
   {\bf g}(x)  &\le & 2\big\{{\bf g}_{1}^{*2}(x)  +   {\bf g}_{1}^{*2} ( x+{1\over 2})  +  {\bf g}_{1}^{*2} (
x-{1\over 2})\big\}=C_0 \Sigma\big[{\bf g}_{1}^{*2}(x) : {\bf I}  \big] \cr
&=&\Sigma\big[\prod^*_{(j,c_j)\in {\bf U} }{\bf g} _{j}^{*c_j} (x) : {\bf I}  \big]. 
  \end{eqnarray} 
 Now we apply $\mathcal E$. We begin with the "stack" of 1's of height $r_1=2$. At first 
 \begin{eqnarray*} 
 {\bf g}_{1}^{*2}(x) &= &\int_\R {\bf g}_{1}(x-y){\bf g}_{1}(y) \dd y=\int_\R {\bf g}_{1}(x-y){\bf g} (2y)
\dd y\cr 
 &\le  &   C_0 \int_\R {\bf g}_{1}(x-y)\Sigma\big[{\bf g}_{1}^{*2}(2y) : {\bf I}  \big] \dd y
  . 
   \end{eqnarray*} 
But by (\ref{Sigma1}), next (\ref{conv0})
 $$\Sigma\big[{\bf g}_{1}^{*2}(2y) : {\bf I}  \big]=\Sigma\big[ T_{1\over 2}({\bf g}_{1}^{*2})  ( y) : {1\over 2}{\bf I} \big]=\Sigma\big[
2 (T_{1\over 2}{\bf g}_{1})^{*2}   ( y) : {1\over 2}{\bf I} \big]=2\Sigma\big[
   {\bf g}_{2} ^{*2}   ( y) : {1\over 2}{\bf I} \big]. $$
Therefore\begin{eqnarray*} 
 {\bf g}_{1}^{*2}(x)  
 &\le  &   2C_0 \int_\R {\bf g}_{1}(x-y)\Sigma\big[ {\bf g}_{2 }^{*2} ( y) : {1\over 2}{\bf I}  \big] \dd y\ 
=2C_0  \Sigma\big[ {\bf g}_{1}*{\bf g}_{2 }^{*2} ( x) : {1\over 2}{\bf I}  \big] . 
   \end{eqnarray*}  By   reporting in (\ref{conv1}), we obtain 
\begin{eqnarray} \label{conv12}
   {\bf g}(x)& \le &   C_0  (2C_0  )\Sigma\Big[\Sigma\big[ {\bf g}_{1}*{\bf g}_{2 }^{*2} ( x)
 : {1\over 2}{\bf I}  \big] : {\bf I} 
\Big]  =  C_1\Sigma\big[ {\bf g}_{1}*{\bf g}_{2 }^{*2} ( x) : {\bf I}_1    \big]\cr &=&C_1\ \Sigma\big[\prod^*_{(j,c_j)\in {\bf U}_1}{\bf g} _{j}^{*c_j}
(x) : {\bf I}_1 \big]. 
  \end{eqnarray}
And $C_1=8$.
 \smallskip\par 
  We now apply $\mathcal E$ once again, and  bound the generic product ${\bf g}_{1}*{\bf g}_{2 }^{*2} ( x)$ by applying (\ref{conv12}) to
${\bf g}_{1}$. Concretely
\begin{eqnarray*} \int_\R {\bf g}_{2 }^{*2}(x-y){\bf g}_{1}(y)\dd y&=&\int_\R {\bf g}_{2 }^{*2}(x-y){\bf g} (2y)\dd y  \cr &\le& C_1
\int_\R {\bf g}_{2 }^{*2}(x-y)\Sigma\big[ {\bf g}_{1}*{\bf g}_{2 }^{*2} ( 2y) : {\bf I}_1  
\big]\dd y  \cr &=& 2^{3-1}C_1 \int_\R {\bf g}_{2 }^{*2}(x-y)\Sigma\big[ {\bf g}_{2 }*{\bf g}_{3}^{*2} (  y) : {1\over 2}{\bf I}_1   
\big]\dd y \cr &=&2^2C_1  \Sigma\big[ {\bf g}_{2 }^{*3}*{\bf g}_{3}^{*2} (  x) : {1\over 2}{\bf I}_1   \big]. 
\end{eqnarray*}   
 
By reporting in (\ref{conv12}), we obtain
\begin{eqnarray} \label{conv13}
   {\bf g}(x) &\le &   
 C_1( 2^ 2C_1)\Sigma\Big[ \Sigma\big[ {\bf g}_{2 }^{*3}*{\bf g}_{3}^{*2} (  x) : {1\over 2}{\bf I}_1    \big] : {\bf I}_1   
\Big]\cr  &= &C_2\Sigma\big[   {\bf g}_{2 }^{*3}*{\bf g}_{3}^{*2} (  x) :     {\bf I}_2   \big]=C_2\ \Sigma\big[\prod^*_{(j,c_j)\in {\bf U}_2}{\bf g}
_{j}^{*c_j} (x) : {\bf I}_2 \big].  
  \end{eqnarray}
And $C_2=256$. 
 For the next $\mathcal E$-iteration, as we have exhausted the stack of $1$'s, we now use
 the stack of $2$'s of height $r_2=3$.   
  We bound the new the generic product
$ {\bf g}_{2 }^{*3}*{\bf g}_{3}^{*2} (  x)
$ by applying (\ref{conv13}) to ${\bf g}_{2 }(x)$  as follows:
 \begin{eqnarray*} \label{divconv}& &\int_\R  {\bf g}_{2 }^{*2}*{\bf g}_{3}^{*2} (  x-y) {\bf g}_{2 }(y)\dd y=\int_\R {\bf g}_{2
}^{*2}*{\bf g}_{3}^{*2} (  x-y) {\bf g}(2^2y)\dd y
 \cr  &\le &C_2\int_\R {\bf g}_{2 }^{*2}*{\bf g}_{3}^{*2} (  x-y)  \Sigma\Big[   {\bf g}_{2 }^{*3}*{\bf g}_{3}^{*2} (  4y)
 :     {\bf I}_3   \Big]\dd y
\cr  &= &2^{2(3+2-1)}C_2\int_\R {\bf g}_{2 }^{*2}*{\bf g}_{3}^{*2} (  x-y) 
\Sigma\big[   {\bf g}_{4}^{*3}*{\bf g}_{5}^{*2} (  y) :     {1\over 4}{\bf I}_2  \big]\dd y
\cr &= &2^8C_2 \ \Sigma\big[  {\bf g}_{2 }^{*2}*{\bf
g}_{3}^{*2}* {\bf g}_{4}^{*3}*{\bf g}_{5}^{*2} (  x) :     {1\over 4}{\bf I}_2   \big] .\end{eqnarray*}
 By reporting in (\ref{conv13}), we obtain
\begin{eqnarray} \label{conv14}
   {\bf g}(x) &\le &   
 2^8C^2_2\ \Sigma\big[ \Sigma\big[  {\bf g}_{2 }^{*2}*{\bf g}_{3}^{*2}* {\bf g}_{4}^{*3}*{\bf g}_{5}^{*2} (  x) :    
{1\over 4}{\bf I}_2   \big] :     {\bf I}_2   \big]\cr  &= &   
 C_3\Sigma\big[    {\bf g}_{2 }^{*2}*{\bf g}_{3}^{*2}* {\bf g}_{4}^{*3}*{\bf g}_{5}^{*2} (  x)   :     {\bf I}_3   
\big], 
  \end{eqnarray}
with $C_3=16777216$. And so on.
\smallskip\par
To simplify, let $k\ge 1$ and $R_{k-1}<m\le R_k$. At step $m$, we play with the stack of $k$'s of height $r_k$ and apply  the bound
previously obtained  to the least dilation of ${\bf g}$ in the generic product  $G= \prod^*_{(j,c_j)\in {\bf U}_{m-1}}{\bf g}
_{j}^{*c_j} (x)$ from the previous step. The dilation factor being
$2^k$, the bound of ${\bf g}_k( x)$   thereby produces the new terms  $T_{2^{-k}} (G)  (  x) =\prod^*_{(j,c_j)\in {\bf U}_{m-1}}{\bf g}
_{j+k}^{*c_j} (x) $. Hence by (\ref{convmany}), after integration,  a constant factor $2^{ k(j_{m-1} -1)}C_{m-1}$.
 Next   we report  the bound obtained for the generic products in the inequality from the preceding step. This is exactly what describes
transform $\mathcal D$.  This generates a new constant factor $ C_{m-1}$. Together with the preceding constant factor,  this gives the
constant
$2^{ k(j_{m-1} -1)}C_{m-1}^2 =C_m$. The rule concerning constants $C_m$ being the same at each step inside the   block $]R_{k-1}, R_k]$,
we have the recurrence relation
\begin{equation} C_m=2^{ k(j_{m-1} -1)}C_{m-1}^2 .\end{equation}
And    the transform $c_j^{m-1}\mapsto c_j^{m }$ is  described by $\mathcal T$.
  \end{proof}
   Let $k\ge 1$. Put
\begin{equation}\label{Rz0} \gamma_{ k} =\sum_{j\in J_{R_k}} c_j^{R_{k}}, \qq d_k =\sum_{j\in J_{R_k}} jc_j^{R_{k}}. 
\end{equation}
 We shall now deduce the following estimate.
 \begin{proposition} \label{p6} Let $\nu$ be a finite     measure such that   $\widehat \nu\ge 0$. Then for any $W>0$
$${1\over 2W}\int_{ -W}^{W}   
 \widehat \nu(t)\dd t \le  C_{R_k}  \,    2^{- d_{ k} +1} \int_\R    \prod_{(j,c_j)\in {\bf
U}_{R_k}} 
\Big[{\sin ({ 2Wx\over  2^{  j}})\over  { 2Wx\over  2^{  j}} }\Big] ^{ c_j}   
 \Big|\sum_{\rho\in 2W{\bf I}_{R_k}}  e^{-i  \rho x } \Big|\nu(\dd x).
$$ 
 \end{proposition}
   \begin{proof} Recall that  $J_{R_k}=\big\{ k+1,\ldots, \zeta_{k } \big\} $.  Further, by (\ref{Rz})
\begin{equation}\label{Rz1} \gamma_{ k} \ge {r_k^2\over 2}(\zeta_{k-1}-2k)={r_k^2\over 2}
\zeta_{k-1}(1-{2k\over\zeta_{k-1}})\ge  ({1-\e \over 2})r_k^2\zeta_{k-1}
  ,  
\end{equation}
once $k\ge k_\e$. 
Similarly 
\begin{equation}\label{Rz2} d_k \ge {r_k^2\over 2}\sum_{j\in   \{ 2k, \ldots, \zeta_{k-1}
\}} j\ge {r_k^2\over 4}(\zeta_{k-1}^2-4k^2)
\ge ({1-\e \over 4} )r_k^2\zeta_{k-1}^2  
\end{equation}
for $k$ large enough.  
 
  By Proposition \ref{p5}, with $ m= R_k$
 $${\bf g}(t) \le C_{R_k}\ \Sigma\big[\prod^*_{(j,c_j)\in {\bf U}_{R_k}}{\bf g} _{j}^{*c_j} (t) : {\bf I}_{R_k} \big] . 
$$
Let 
$ W>0$.    Then by (\ref{Sigma1}), next   (\ref{convmany})
\begin{eqnarray} 
 \chi_{[-W,W]}(t)
&=&   g( {t\over 2W} )  
 \le C_{R_k}\ \Sigma\big[\prod^*_{(j,c_j)\in {\bf U}_{R_k}}{\bf g} _{j}^{*c_j} ({t\over 2W}) : {\bf I}_{R_k} \big]
\cr & =& C_{R_k}\ \Sigma\big[T_{2W}\Big(\prod^*_{(j,c_j)\in {\bf U}_{R_k}}{\bf g} _{j}^{*c_j}\Big) ( t) : 2W{\bf I}_{R_k}
\big]
 \cr & =&     { C_{R_k}\over (2W)^{\gamma_{ k}-1}}  \ \Sigma\big[ \prod^*_{(j,c_j)\in {\bf U}_{R_k}}     (T_{2W}{\bf g}
_{j})^{*c_j}    ( t) : 2W{\bf I}_{R_k}
\big].
\end{eqnarray}
   
 By (\ref{Sigma11})
\begin{equation} \widehat{\Sigma }\big[ \prod^*_{(j,c_j)\in {\bf U}_{R_k}}     (T_{2W}{\bf g} _{j})^{*c_j}     : 2W{\bf
I}_{R_k}
\big](x)
 = \prod_{(j,c_j)\in {\bf U}_{R_k}}\widehat {  T_{2W}{\bf g} _{j}}   
 (x)^{ c_j}\ \Big\{\sum_{\rho\in {2W}{\bf I}_{R_k}}  e^{-i  \rho x } \Big\} . 
\end{equation} 
But 
$$ \widehat {  T_{2W}{\bf g} _{j}}(x)= \int_\R e^{ixu}{\bf g}({{2^{  j}u\over 2W}         }   )\dd u
={2W\over  2^{  j}} \int_\R e^{ {i2W \over  2^{  j}}x v}{\bf g}(v  )\dd  v
={2W\over  2^{  j}}  \, \widehat {  {\bf g} }({ 2Wx\over  2^{  j}}).$$
  Hence 
\begin{eqnarray}& & \widehat{\Sigma }\big[ \prod^*_{(j,c_j)\in {\bf U}_{R_k}}     (T_{2W}{\bf g} _{j})^{*c_j}     : 2W{\bf
I}_{R_k}
\big](x)\cr 
 &=& (2W)^{ \gamma_{ k}} 2^{- d_{k}} \prod_{(j,c_j)\in {\bf U}_{R_k}}  \widehat {  {\bf g} }\big({ 2Wx\over  2^{  j}}\big)^{ c_j}\
\Big\{\sum_{\rho\in 2W{\bf I}_{R_k}}  e^{-i  \rho x } \Big\} . 
\end{eqnarray}

 And by the Parseval relation
 \begin{eqnarray} & &{1\over 2W}\int_{ -W}^{W}   
 \widehat \nu(t)\dd t \le        {C_{R_k}\over (2W)^{\gamma_{ k} }} \   \int_{\R}      \Sigma\big[
\prod^*_{(j,c_j)\in {\bf U}_{R_k}}     (T_{2W}{\bf g} _{j})^{*c_j}    ( t) : 2W{\bf I}_{R_k}
\big]
\widehat \nu(t)\dd t  \,  
 \cr &= &{C_{R_k}\over (2W)^{\gamma_{ k} }}\,  (2W)^{ \gamma_{ k}} 2^{- d_{k}}  \int_\R    \prod_{(j,c_j)\in
{\bf U}_{R_k}} 
\widehat {  {\bf g} }({ 2Wx\over  2^{  j}}) ^{ c_j}   
\Big\{\sum_{\rho\in 2W{\bf I}_{R_k}}  e^{-i  \rho x } \Big\}\nu(\dd x)
\cr &= &C_{R_k}  \,    2^{- d_{ k} +1} \int_\R    \prod_{(j,c_j)\in {\bf
U}_{R_k}} 
\Big({\sin ({ 2Wx\over  2^{  j}})\over  { 2Wx\over  2^{  j}} }\Big) ^{ c_j}   
 \Big\{\sum_{\rho\in 2W{\bf I}_{R_k}}  e^{-i  \rho x } \Big\}\nu(\dd x).
\cr &  &  \end{eqnarray}
\end{proof} 
\section{Proof of Theorem \ref{cc}}  By assumption
$\nu\ge 0$.   Choose 
  $    W = 2^{\zeta_k }T $. Then
$${\sin ({ 2Wx\over  2^{  j}})\over  {2W x\over  2^{  j}} }= {\sin ({  2^{\zeta_k+1-j}xT })\over 2^{\zeta_k+1-j}xT }$$
But we have 
 that 
\begin{equation} \big|\sin \sum_{k=1}^n x_k\big|\le \sum_{k=1}^n \sin x_k,
\end{equation}
if $0<x_k<\pi$ and $n>1$, see \cite{M} p.236. From this easily follows that $|\sin nx|\le n|\sin x|$ for any real $x$ and any integer
$n$. Indeed, write $x=x'+k\pi$ with $0<x'<\pi$. Then $|\sin n x|=|\sin (nx' +nk\p)|=|\sin  nx'  |\le
n|\sin  nx'  |= n|\sin  nx   |
$.

 Consequently
\begin{equation}     \Big|{\sin  {  2^{\zeta_k+1-j}xT } \over 2^{\zeta_k+1-j}xT } \Big|\le \Big|{\sin  {  xT } \over
 xT } \Big|.
\end{equation}

  By reporting and since $ \#\{{\bf I}_{R_k} \} = 3^{R_k}$ we get
\begin{equation}{1\over 2.2^{\zeta_k } T}\int_{ -2^{\zeta_k } T}^{2^{\zeta_k } T }   
\widehat \nu(\dd t)    \le   C_{R_k}    3^{R_k}    2^{- d_{ k} }   
\int_\R      
 \Big|{\sin  {  xT } \over
 xT } \Big|^{ \gamma_{ k}}   
 \nu(\dd x).
 \end{equation}
   
And by using estimates (\ref{Rz1}), (\ref{Rz2})
\begin{equation}{1\over 2.2^{\zeta_k } T}\int_{ -2^{\zeta_k } T}^{2^{\zeta_k } T }   
\widehat \nu(\dd t)    \le   C_{R_k}   3^{R_k}    2^{ - {1 \over 5}r_k^2\zeta_{k-1}^2}    
\int_\R      
 \Big|{\sin  {  xT } \over
 xT } \Big|^{{1 \over 3}r_k^2\zeta_{k-1}}   
 \nu(\dd x).
 \end{equation} 
We now   estimate $C_{R_k}$. 
 By iterating inside the block of integers $]R_{k-1}, R_k]$ the recurrence relation   $C_m=2^{ k(j_{m-1} -1)}C_{m-1}^2$ obtained in
Proposition \ref{p5}, we obtain 
$$C_{R_{k }}=2^{k\{(j_{R_{k-1}}-1)+\ldots+(j_{R_{k }}-1)\}} C_{R_{k-1}}^{2^{r_k}}. $$ 
According to (\ref{est0}), we have 
 \begin{eqnarray*}k\{(j_{R_{k-1}}-1)+\ldots+(j_{R_{k }}-1)\}&=&k \sum_{R_{k-1}< m\le R_k} (j_m-1)\cr&=&kr_k(\zeta_{k-1}-1)+k^2 
 {r_k (r_k+1)\over 2}.   
 \end{eqnarray*}
As $\zeta_k=1+r_1+2r_2+\ldots+kr_k$, it follows that 
 $$kr_k(\zeta_{k-1}-1)+k^2 
 {r_k (r_k+1)\over 2}\le kr_k \zeta_{k-1} +k^2 
  r_k^2 \le  \zeta_k \zeta_{k-1} +\zeta_k^2\le 2\zeta_k^2.  $$

Thus
\begin{equation}\label{recrk} C_{R_{k }} \le 2^{2\zeta_k^2}\ C_{R_{k-1}}^{2^{r_k}}. 
\end{equation} 
 By successively iterating this,  and since $C_{R_{1}}=2$,  we get
\begin{eqnarray*} C_{R_{k }}&\le& 2^{2\{\zeta_k^2+\zeta_{k-1}^22^{r_k}+
\zeta_{k-2}^2(2^{r_k}+2^{r_{k-1}})+\ldots
+\zeta_{2}^2(2^{r_k}+\ldots +2^{r_{3}}))\}}2^{2^{r_k}+2^{r_{k-1}}+ \ldots+2^{r_{ 2}}}
\cr 
&\le& 2^{2 .2^{r_k} k\zeta_k^2 } .
\end{eqnarray*}
But   $ r_{k}\ge \rho^k$ by (\ref{roestim}), so that
 $$R_k\le \zeta_k=1+r_1+2r_2+\ldots+kr_k \ll_\e 2^{\e r_k}.$$Hence also
\begin{equation} C_{R_{k }}\le 2^{ 2^{(1+\e)r_k} }.
\end{equation}

Finally, 
\begin{eqnarray}{1\over   T}\int_{ -2^{\zeta_k } T}^{2^{\zeta_k } T }   
\widehat \nu(\dd t)   & \ll_\e &  2^{ 2^{(1+\e)r_k} }2^{ \zeta_k  - {1 \over 5}r_k^2\zeta_{k-1}^2   }       
\int_\R      
 \Big|{\sin  {  xT } \over
 xT } \Big|^{{1 \over 3}r_k^2\zeta_{k-1}}   
 \nu(\dd x)
\cr 
 & \ll_\e &  2^{ 2^{(1+\e)r_k} }        
\int_\R      
 \Big|{\sin  {  xT } \over
 xT } \Big|^{ r_k^2 }   
 \nu(\dd x).
 \end{eqnarray}
Thereby, since $\zeta_k\ge r_k$
\begin{equation} \int_{ -2^{r_k } T}^{2^{r_k } T }   
\widehat \nu(\dd t)    \ll_\e  T\,2^{2^{(1+\e)r_k}}       
\int_\R      
 \Big|{\sin  {  xT } \over
 xT } \Big|^{ r_k^2 }   
 \nu(\dd x).
 \end{equation}

{ } \noi
 
  \end{document}